\documentclass[12pt]{amsart}
\usepackage{latexsym, amsmath, amssymb,srcltx}
\usepackage{epsfig}
\usepackage{amscd}
\usepackage{pstricks,pst-node,cite}

\newtheorem{thm}{Theorem}[section]
\newtheorem{lem}[thm]{Lemma}

\newtheorem{cor}[thm]{Corollary}

\theoremstyle{remark}
\newcommand{\C}[1]{\mathcal #1}
\newcommand{\B}[1]{\mathbb #1}

\def \Aut {\mbox{\rm Aut}}

\def \Fix{\mbox{\rm Fix}}
\def \Inn {\mbox{\rm Inn}}

\ifx\blackandwhite\undefined
  \psset{linecolor=blue}\newrgbcolor{darkred}{.9 0 0}\newrgbcolor{emgreen}{0 .9 0}
  \newrgbcolor{lightgreen}{.7  .99 .7}\newrgbcolor{lightred}{.99 .7 .7} \else
  \psset{linecolor=blue}\newgray{darkred}{.7}\newgray{emgreen}{0.3}\newgray{pup}{.5}\newgray{xxx}{.5}
\fi
 \psset{linewidth=1pt,dash=3pt 3pt,doublesep=.05,dotsize=1pt 5}
\SpecialCoor

\begin{document}

\author{Dongseok Kim}
\address{Department of Mathematics \\Kyonggi University
\\ Suwon, 443-760 Korea}
\email{dongseok@kgu.ac.kr}
\thanks{}

\author{Young Soo Kwon}
\address{Department of Mathematics \\Yeungnam University \\Kyongsan, 712-749, Korea}
\email{ysookwon@ynu.ac.kr}

\author{Jaeun Lee}
\address{Department of Mathematics \\Yeungnam University \\Kyongsan, 712-749, Korea}
\email{julee@yu.ac.kr}
\thanks{???}

\subjclass[2000]{Primary 05C30; Secondary 05C25} \keywords{Cayley
graphs, weak equivalences, equivalences, circulant graphs}

\title[Degree distributions for a class of Circulant graphs]{Degree distributions for a class of Circulant graphs}
\begin{abstract}
We characterize the equivalence and the weak equivalence of Cayley
graphs for a finite group $\C{A}$. Using these characterizations, we
find degree distribution polynomials for weak equivalence of some graphs including 1) circulant graphs
of prime power order, 2) circulant graphs
of order $4p$, 3) circulant graphs
of square free order and 4) Cayley graphs of order $p$ or $2p$. 
As an application, we find an enumeration formula for the number of weak equivalence classes
of circulant graphs of prime power order, order $4p$ and square free order and Cayley graphs of order $p$ or $2p$.
\end{abstract}

\maketitle

\section{Introduction}

Let $\C{A}$ be a finite group with identity $e$ and let $\Omega$ be
a generating set for $\C{A}$ with properties that $\Omega =
\Omega^{-1}$ and $e\not\in \Omega$, where $\Omega^{-1}=\{x^{-1}\,|\,
x\in \Omega\}$. The \emph{Cayley graph} $C(\C{A}, \Omega)$ is a
simple graph whose vertex set and edge set are defined as follows:
$$V(C(\C{A}, \Omega))=\C{A}~\mathrm{and}~ E(C(\C{A}, \Omega))=\{\{g, h\}\,
|\, g^{-1}h \in \Omega\}.$$ Because of their rich connections with a
broad range of areas, Cayley graphs have been in the center of the research in
graph theory~\cite{BMPRS, DSS, PR, Rosenhouse}. Spectral estimations
of Cayley graphs have been studied~\cite{Cioaba, FMT}. It plays a
key role in the study of the geometry of hyperbolic
groups~\cite{kapo}. Recently, Li has found wonderful results on
edge-transitive Cayley graphs~\cite{Li, LL}. For standard terms and notations, we refer
to~\cite{GTP1}.

The Cayley graph $C(\C{A}, \Omega)$ admits a natural $\C{A}$-action,
$\cdot: \C{A} \times C(\C{A}, \Omega) \to C(\C{A}, \Omega)$ defined
by $g\cdot g'=gg'$ for all $g,g'\in\C{A}$. A graph $\Gamma$ with an
$\C{A}$-action is called an \emph{$\C{A}$-graph}. So, every Cayley
graph $C(\C{A}, \Omega)$ is an $\C{A}$-graph. A graph isomorphism
$f:\Gamma_1 \to \Gamma_2$ between two $\C{A}$-graphs $\Gamma_1$ and $\Gamma_2$ is a \emph{weak equivalence} if
there exists a group automorphism $\alpha:\C{A}\to \C{A}$ such that
$f(g\cdot u) = \alpha(g)\cdot f(u)$ for all $g\in \C{A}$ and $u\in
V(\Gamma_1)$. When $\alpha$ is the identity automorphism, we say that $f$
is an \emph{equivalence}. If there is a weak equivalence between
$\C{A}$-graphs $\Gamma_1$ and $\Gamma_2$, we say $\Gamma_1$ and $\Gamma_2$ are \emph{weak
equivalent}. Similarly, if there is an equivalence between
$\C{A}$-graphs $\Gamma_1$ and $\Gamma_2$, we say $\Gamma_1$ and $\Gamma_2$ are
\emph{equivalent}. Enumerations of the equivalence classes and weak equivalence classes
of some graphs have been studied~\cite{FKKL, KL1}.

In particular, the isomorphism problem of Cayley
graphs has been studied by several authors~\cite{CL, FL, Li98}.
However, one classical isomorphism problem on Cayley graphs is a
conjecture rose by $\acute{{\rm A}}$d$\acute{{\rm a}}$m \cite{Adam}
that two Cayley graphs of $\B{Z}_n$ are isomorphic
if and only if they are isomorphic by a group automorphism of $\B{Z}_n$. This conjecture
was first disproven by Elpas and Turner~\cite{ET70}. After that, a particular attention
has been paid to determine which group $\C{A}$ has the property that
two Cayley graphs of $\C{A}$ are isomorphic if and only if they are isomorphic by
a group automorphism of $\C{A}$. Such a group is called a \emph{CI-group}.
For cyclic group $\B{Z}_n$, CI-groups were completely classified by Muzychuk
\cite{Muzy, Muzy1} that a cyclic group
of order $n$ is a CI-group if and only if $n = 8, 9, 18, m, 2m$ or $4m$ where $m$ is odd and
square-free. Therefore, for any square-free number $n$, the number of
non-isomorphic connected $k$-regular Cayley graphs of $\B{Z}_n$ is
equal to the coefficient $a_k (\B{Z}_{n})$ of $x^k$ in $\Psi_{\B{Z}_{n}}^w(x)$
which is the degree distribution polynomial of weak equivalence classes of Caylay graphs whose underlying group is $\C{\B{Z}_{n}}$
and the number of non-isomorphic connected Cayley graphs of $\B{Z}_n$ is equal to $\Psi_{\B{Z}_{n}}^w(1)$.
Li has a wonderful survey for the isomorphism problem of Cayley
graphs~\cite{Li02}.

A \emph{circulant} graph is a graph whose automorphism group
of the graph includes a cyclic subgroup which acts transitively on 
vertex set of the graph.
The isomorphism problem of circulant
graphs had been studied by several authors~\cite{AP, Dobson} and completely
solved by Muzychuk~\cite{Muzy2}.

In present article, we deal with the weak equivalence classes of circulant graphs.
We first find a characterization of the equivalence and the weak equivalence of Cayley
graphs for a finite group $\C{A}$.
As the main result of the article, we find the degree distribution polynomials
for the weak equivalence classes of circulant graphs of prime power order or square free order.
As an application, we find an enumeration formula for the number of the weak equivalence classes
of circulant graphs of prime power order, order $4p$ and square free order and
the number of the weak equivalence classes
of Cayley graphs of order $p$ or $2p$.

The outline of this paper is as follows. In Section~\ref{char}, we
characterize the equivalence and the weak equivalence of Cayley graphs for a finite
group $\C{A}$. In Section~\ref{abdn}, we find some computation formulae for degree distribution polynomials.
Combining results in these two sections,
we find the degree distribution polynomials for the weak equivalence classes of circulant graphs of prime power order
and circulant graphs of square free order in Section~\ref{DistC}.
At last we find the degree distribution polynomials for the weak equivalence classes of Cayley graphs of order $p$ or $2p$ in
Section~\ref{cayleydis}.

\section{A characterization of Cayley graphs}
\label{char}

Our definition of a weak equivalence between two Cayley graphs can
be interpreted as a color-consistence and a direction preserving
graph isomorphism~\cite[Section 1.2.4]{GTP1}.

\begin{thm}\label{weakequiv}
Let  $C(\C{A}, \Omega)$ and $C(\C{A}, \Omega')$ be two Cayley
graphs. The followings are equivalent.
\begin{enumerate}
\item[{\rm (1)}] $C(\C{A}, \Omega)$ and $C(\C{A}, \Omega')$
           are weakly equivalent,
\item[{\rm (2)}] There exists  an automorphism $\alpha:\C{A}\to
\C{A}$ such that $\alpha(\Omega)=\Omega'$.
\end{enumerate}
\end{thm}

\begin{proof}
$(1)\Rightarrow(2)$: Let $f:C(\C{A}, \Omega)\to C(\C{A}, \Omega')$
be a weak equivalence. Then there exists a group automorphism
$\tau:\C{A}\to \C{A}$ such that $f(g)=\tau(g)f(e)$ for each $g\in
\C{A}$. Let $x\in \Omega$. Then $\{e, x\}$ is an edge in $C(\C{A},
\Omega)$. Since $\{f(e), f(x)\}$ is an edge in
 $C(\C{A}, \Omega')$, $f(e)^{-1}f(x)=f(e)^{-1}\tau(x)f(e)$ is
 an element of $\Omega'$. Hence the map $\alpha:\C{A}\to\C{A}$
 defined by $\alpha(g)=f(e)^{-1}\tau(g)f(e)$ is a group
 automorphism such that $\alpha(\Omega)=\Omega'$.

$(2)\Rightarrow(1)$: Let  $\alpha:\C{A}\to \C{A}$ be a group
automorphism such that $\alpha(\Omega)=\Omega'$. We define
$f:C(\C{A}, \Omega)\to C(\C{A}, \Omega')$ by $f(g)=\alpha(g)$. If
$\{g,h\}$ is an edge in $C(\C{A}, \Omega)$, then $g^{-1}h\in \Omega$
and
$f(g)^{-1}f(h)=\alpha(g)^{-1}\alpha(h)=\alpha(g^{-1}h)\in\alpha(\Omega)=\Omega'$.
Hence $f$ is a graph isomorphism such that
$f(gg')=\alpha(gg')=\alpha(g)\alpha(g')=\alpha(g)f(g')$, $i.e.$, $f$
is a weak equivalence.
\end{proof}

By using a similar method in the proof of Theorem~\ref{weakequiv}, we can have
the following theorem.

\begin{thm}\label{equiv}
Let  $C(\C{A}, \Omega)$ and $C(\C{A}, \Omega')$ be
two Cayley graphs. The followings are equivalent.
\begin{enumerate}
\item[{\rm (1)}] $C(\C{A}, \Omega)$ and $C(\C{A}, \Omega')$
           are  equivalent,
\item[{\rm (2)}] $\Omega$ and $\Omega'$ are conjugate in $\C{A}$,
$i.e.$, there exists an element $\gamma\in \C{A}$ such that $\gamma^{-1}\Omega
\gamma=\Omega'$.
\end{enumerate}
\begin{proof}
$(1)\Rightarrow(2)$: Let $f:C(\C{A}, \Omega)\to C(\C{A}, \Omega')$
be an equivalence and let $x\in \Omega$. Then $\{e, x\}$ is an edge in $C(\C{A},
\Omega)$. Since $\{f(e), f(x)\}$ is an edge in
 $C(\C{A}, \Omega')$, $f(e)^{-1}f(x)=f(e)^{-1}xf(e)$ is
 an element of $\Omega'$. Hence  $f(e)^{-1}\Omega f(e)=\Omega'$.

$(2)\Rightarrow(1)$: Let $\gamma$ be an element of $\C{A}$ such that
$\gamma^{-1}\Omega\gamma=\Omega'$.
 We define
$f:C(\C{A}, \Omega)\to C(\C{A}, \Omega')$ by $f(g)=g\gamma$. If
$\{g,h\}$ is an edge in $C(\C{A}, \Omega)$, then $g^{-1}h\in \Omega$
and
$f(g)^{-1}f(h)=(g\gamma)^{-1}(h\gamma)=(\gamma^{-1}g^{-1})(h\gamma)=\gamma^{-1}(g^{-1}h)\gamma\in \gamma^{-1}\Omega\gamma=\Omega'$.
Hence $f$ is a graph isomorphism such that
$f(gg')=(gg')\gamma=g(g'\gamma)=gf(g')$, $i.e.$, $f$ is an equivalence.
\end{proof}
\end{thm}

For a finite group $\C{A}$, let  $G(\C{A})=\{\Omega\subset
\C{A}\,:\, \Omega^{-1}=\Omega, <\Omega>=\C{A}, e\not\in
\Omega \}.$ Notice that $G(\C{A})$ contains all equivalence
classes of Cayley graphs $C(\C{A},\Omega)$. Furthermore, any subgroup of group automorphisms of $\C{A}$
 admits a natural action on $G(\C{A})$ by $\alpha\cdot
\Omega=\alpha(\Omega)$. By Theorem~\ref{weakequiv},
$\C{E}^w(\C{A})$, the number of weak equivalence classes of Cayley
graphs $C(\C{A},\Omega)$, is equal to the number of
orbits of the $\Aut(\C{A})$ action on $G(\C{A})$, where
$\Aut(\C{A})$ is the group of all group isomorphisms of $\C{A}$.
Similarly, one can see that the number $\C{E}(\C{A})$ of the
equivalence classes of Cayley graphs $C(\C{A},\Omega)$
is equal to the number of orbits of the $\Inn(\C{A})$ action on
$G(\C{A})$ by Theorem~\ref{equiv}, where $\Inn(\C{A})$ is the
group of all inner automorphisms of $\C{A}$.

For a finite group $\C{A}$, let $a^w_k(\C{A})$ (resp., $a_k(\C{A})$) be the number of the weak equivalence
(resp., equivalence) classes of Cayley graphs $C(\C{A}, \Omega)$ with degree $k$.
We call the polynomial
$$\Psi_\C{A}^w(x)=\sum_{k=1}^{|\C{A}|-1} a^w_k(\C{A})x^k~~
{\rm{(resp.}}, \Psi_\C{A}(x)=\sum_{k=1}^{|\C{A}|-1} a_k(\C{A})x^k)$$
the \emph{degree distribution polynomial of the weak equivalence
classes} (equivalence classes, respectively) of Cayley graphs
whose underlying group is $\C{A}$. Notice that $\Psi_\C{A}^w(1) =\C{E}^w(\C{A})$
and  $\Psi_\C{A}(1) =\C{E}(\C{A})$. For convenience, for any finite group $\C{A}$ and any automorphism
$\alpha\in \Aut(\C{A})$, let $\Fix_\alpha(\C{A})=\{\Omega \in G(\C{A}) : \alpha(\Omega)=\Omega\}$.
Now the following theorem comes from the Burnside lemma.

\begin{thm}\label{formul1}
Let $\C{A}$ be a finite group. Then we have
$$\Psi_\C{A}^w(x)=\frac{1}{|\Aut(\C{A})|} \sum_{\alpha\in {\rm{Aut}}(\C{A})}\left(\sum_{\Omega\in {\rm{Fix}}_\alpha(\C{A})} x^{|\Omega|}\right),$$
and
$$\Psi_\C{A}(x)=\frac{1}{|\Inn(\C{A})|}  \sum_{\alpha\in {\rm{Inn}}(\C{A})}\left(\sum_{\Omega\in {\rm{Fix}}_\alpha(\C{A})} x^{|\Omega|}\right).$$
\end{thm}

In order to compute  $\sum_{\Omega\in {\rm{Fix}}_\alpha(\C{A})} x^{|\Omega|}$, we will find a formula in terms of the M\"obius function defined on the
subgroup lattice of ${\C{A}}$.  The M\"obius function assigns
an integer $\mu (K)$ to each subgroup $K$ of ${\C{A}}$ by the
recursive formula $$\sum_{H\ge K} \mu
(H)=\delta_{K,{\C{A}}}=\left\{\begin{array}{l} 1\
\,~\mbox{if}\
K={\C{A}},\\[.5ex] 0\ \,~\mbox{if}\ K<{\C{A}}.
\end{array}
\right.$$
Jones~\cite{Jo95, Jo99} used such  functions to  count
the normal subgroups of a surface
group and a crystallographic group, and applied it to count certain
covering surfaces.

For convenience, let $S(\C{A})=\{\Omega\subset \C{A}:\Omega=\Omega^{-1}, e\not\in\Omega \}$
for a finite group $\C{A}$ and for any subgroup $K$ of $\C{A}$,
let $S(K)=\{\Omega\subset K:\Omega=\Omega^{-1}, e\not\in\Omega \}$. Then we can see that
$S(\C{A})=\bigcup_{K\leq {\C{A}}} G(K),$
and that
$$\sum_{\Omega\in S(\C{A}),\,\alpha(\Omega)=\Omega}x^{|\Omega|}=
\sum_{K\leq {\C{A}}} \left(\sum_{\Omega\in G(K),\,\alpha(\Omega)=\Omega}x^{|\Omega|}\right)
=\sum_{K\leq {\C{A}}} \left(\sum_{\Omega\in{\rm{Fix}}_\alpha(K) }x^{|\Omega|}\right) .$$
Now, the following lemma easily can be obtained by the M\"obius inversion.

\begin{lem}\label{fix}
Let $\C{A}$ be a finite group and let $\alpha\in \Aut(\C{A})$. Then
$$\sum_{\Omega\in {\rm{Fix}}_\alpha(\C{A})} x^{|\Omega|}
=\sum_{K\le \C{A}} \mu(K) \left(\sum_{\Omega\in S(K),\,\alpha(\Omega)=\Omega}x^{|\Omega|}\right).$$
\end{lem}

Using Lemma~\ref{fix}, we can rephrase Theorem~\ref{formul1} as follows.

\begin{thm}\label{formul2}
Let $\C{A}$ be finite group. Then we have
$$\Psi_\C{A}^w(x)=\frac{1}{|\Aut(\C{A})|} \sum_{\alpha\in {\rm{Aut}}(\C{A})}
\left(\sum_{K\le \C{A}} \mu(K) \left(\sum_{\Omega\in S(K),\,\alpha(\Omega)=\Omega}x^{|\Omega|}\right)\right),$$
and
$$\Psi_\C{A}(x)=\frac{1}{|\Inn(\C{A})|}  \sum_{\alpha\in {\rm{Inn}}(\C{A})}
\left(\sum_{K\le \C{A}} \mu(K) \left(\sum_{\Omega\in S(K),\,\alpha(\Omega)=\Omega}x^{|\Omega|}\right)\right).$$
\end{thm}

\section{Distribution for equivalence classes} \label{abdn}

In this section, we will find a computation formula for  the polynomial $\Psi_\C{A}(x)$
when $\C{A}$ is a finite abelian group or the dihedral group $D_n$ of order $2n$.
If $\C{A}$ is abelian, then $\Inn(\C{A})$ is trivial and
$$\Psi_{\C{A}}(x)=\sum_{K\le \C{A}} \mu(K) \left(\sum_{\Omega\in S(K)}x^{|\Omega|}\right),$$
by Theorem~\ref{formul2}.
It is not hard to show that
$$\sum_{\Omega\in S(K)}x^{|\Omega|}=\left((1+x^2)^{\frac{|K|-|O_2(K)|-1}{2}}
(1+x)^{|O_2(K)|}-1\right),$$
where $O_2(K)=\{g \in K : g^2=e, g\not= e\}$.
We summarize our discussion as follows.

\begin{thm}\label{eqabel}
For a finite abelian group $\C{A}$,
 $$ \Psi_{\C{A}}(x) = \sum_{K\le \C{A}} \mu(K) \left((1+x^2)^{\frac{|K|-|O_2(K)|-1}{2}}
(1+x)^{|O_2(K)|}-1\right),$$
and hence
$$\C{E}(\C{A}) = \Psi_{\C{A}}(1) = \sum_{K\le \C{A}} \mu(K) \left(2^{\frac{|K|+|O_2(K)|-1}{2}}
-1\right).$$
\end{thm}

Let $n$ be a positive integer and $\mu$ be the number theoretical mu-function.
Since a subgroup of the cyclic group $\B{Z}_n$ is also cyclic, say $\B{Z}_m$ with
$m|n$ and that $\mu(\B{Z}_m)=\mu(\frac{n}{m})$. Since
$|O_2(\B{Z}_m)|={\frac{1+(-1)^m}{2}}$,
we have the following corollary from Theorem~\ref{eqabel}.

\begin{cor}\label{eqcyclic}
For any positive integer $n$,
 $$ \Psi_{\B{Z}_n}(x) = \sum_{d|n} \mu\left(\dfrac{n}{d}\right) \left((1+x^2)^{\lfloor \frac{d-1}{2} \rfloor}
(1+x)^{\frac{1+(-1)^d}{2}}-1 \right),$$
and hence
$$\C{E}(\B{Z}_n) = \Psi_{\B{Z}_n}(1) = \sum_{d|n} \mu\left(\dfrac{n}{d}\right) \left(2^{\lfloor \frac{d}{2} \rfloor}
-1 \right).$$
\end{cor}

Now we aim to compute $\Psi_{\B{D}_n}(x)$, where $\B{D}_n=\{a, b: a^n=1, b^2=1, bab=a^{-1}\}$ is the dihedral group  of order $2n$.
Notice that $\Inn(\B{D}_n)$ is the set $\{\alpha_k, \beta_k: k=1,2,\ldots, n\}$, where $\alpha_k(a)=a^{-k}aa^k=a$, $\alpha_k(b)=a^{-k}ba^k=ba^{2k}$,
and $\beta_k(a)=(ba^k) a  (ba^k)=a^{-1}$, $\beta_k(b)=(ba^k) b  (ba^k)=ba^{2k}$
for each $ k=1,2,\ldots, n$.
Each  subgroup of the dihedral group $\B{D}_n$ is isomorphic to either
$\B{Z}_m$ or $\B{D}_m$ for some $m|n$.  There are exactly $\frac{n}{m}$ subgroups
isomorphic to $\B{D}_m$ and only one subgroup isomorphic to $\B{Z}_m$, where
$\B{D}_1$ is the subgroup generated by a reflection and $\B{D}_2$ is a subgroup isomorphic to
$\B{Z}_2\oplus \B{Z}_2$. Moreover, $\mu(\B{Z}_m)=-\frac{n}{m} \mu(\frac{n}{m})$,
$\mu(\B{D}_m)= \mu(\frac{n}{m})$, and  for each $m|n$ we can list the  $\frac{n}{m}$ subgroups
isomorphic to $\B{D}_m$ as follows: for each fixed $\ell=1,2,\ldots, \frac{n}{m}$,
$\B{D}_m(\ell)=\{a^{s\frac{n}{m}},ba^{s\frac{n}{m}+\ell}: s=1,2,\ldots,m\}$.
Then for each divisor $m$ of $n$, we can see that

\begin{align*}
&\sum_{\Omega\in S(\B{D}_m(\ell)),\,\alpha_k(\Omega)=\Omega}x^{|\Omega|}\\
&= \begin{cases}
(1+x^2)^{\alpha}
(1+x)^\frac{1+(-1)^m}{2} \left(1+x^{\frac{m}{\left(\frac{2km}{n},~ m\right)}}\right)^{\left(\frac{2km}{n},~ m\right)} -1
 & \mbox{ if $2k\equiv 0$ (mod  $\frac{n}{m}$),}\\
(1+x^2)^{\alpha}
(1+x)^\frac{1+(-1)^m}{2} -1  & \mbox{ otherwise,}
\end{cases} \end{align*}
and that
\begin{align*}
& \sum_{\Omega\in S(\B{D}_m(\ell)),\,\beta_k(\Omega)=\Omega}x^{|\Omega|}\\
&= \begin{cases}
(1+x^2)^{\alpha} (1+x)^\frac{1+(-1)^m}{2}
(1+x^2)^{\alpha} (1+x)^\frac{3+(-1)^m}{2}  -1
 & \mbox{ if $2(k-\ell)\equiv 0$ (mod  $\frac{n}{m}$),}\\
(1+x^2)^{\alpha}
(1+x)^\frac{1+(-1)^m}{2} -1  & \mbox{ otherwise,}
\end{cases} \end{align*}
where $\alpha = \lfloor \frac{m-1}{2} \rfloor$.

Similarly, we can see that
\begin{align*}
\sum_{\Omega\in S(\B{Z}_m),\,\alpha_k(\Omega)=\Omega}x^{|\Omega|}
&=\sum_{\Omega\in S(\B{Z}_m),\,\beta_k(\Omega)=\Omega}x^{|\Omega|}
=    (1+x^2)^{\lfloor \frac{m-1}{2} \rfloor}
(1+x)^\frac{1+(-1)^m}{2} -1\\
&= \begin{cases} (1+x^2)^{\frac{m-1}{2}}-1 & \mbox{ if $m$ is odd,}\\
  (1+x^2)^{\frac{m-2}{2}}(1+x)  -1
& \mbox{ if $m$ is even.}
\end{cases} \end{align*}

Now, Theorem~\ref{eqdn} follows from the above discussion and Theorem~\ref{formul2}.

\begin{thm}\label{eqdn}
Let $\B{D}_n$ be the dihedral group of order $2n$. Then we have
\begin{align*} 2n \Psi_{\B{D}_n}(x) =&
  \sum_{k=1}^n
\sum_{m|n} \left[-2\frac{n}{m}~\mu\left(\frac{n}{m}\right)
\left( (1+x^2)^{\gamma}
(1+x)^\frac{1+(-1)^m}{2} -1\right)\right.\\
&  +\left.  \mu\left(\frac{n}{m}\right) \sum_{\ell=1}^\frac{n}{m}
\left(F_{\alpha_k}(\ell, m) + F_{\beta_k}(\ell, m)\right)\right],
\end{align*}
where
\begin{align*}
& F_{\alpha_k}(\ell, m)\\
&= \begin{cases}  (1+x^2)^{\gamma}
(1+x)^\frac{1+(-1)^m}{2} \left(1+x^{\frac{m}{\left(\frac{2km}{n}, m\right)}}\right)^{\left(\frac{2km}{n}, m\right)} -1
 & \mbox{\rm if $2k\equiv0$ (mod  $\frac{n}{m}$),}\\
 (1+x^2)^{\gamma}
(1+x)^\frac{1+(-1)^m}{2} -1  & \mbox{\rm otherwise,}
\end{cases} \end{align*}
and
\begin{align*}
& F_{\beta_k}(\ell, m)\\[1ex]
&= \begin{cases} (1+x^2)^{\gamma} (1+x)^\frac{1+(-1)^m}{2}
(1+x^2)^{\gamma} (1+x)^\frac{3+(-1)^m}{2}  -1
 & \mbox{\rm if $2(k-\ell)\equiv0$ (mod  $\frac{n}{m}$),}\\
  (1+x^2)^{\gamma}
(1+x)^\frac{1+(-1)^m}{2} -1  & \mbox{\rm otherwise,}
\end{cases} \end{align*}
where $\gamma = \lfloor \frac{m-1}{2} \rfloor$.
\end{thm}

\begin{cor}\label{eqdncount}
Let $\B{D}_n$ be the dihedral group of order $2n$. Now we have
 $$\C{E}(\B{D}_n) = \Psi_{\B{D}_n}(1) = \sum_{m|n} (2, \frac{n}{m}) \mu\left(\frac{n}{m}\right) 2^{\lfloor \frac{m}{2} \rfloor}
\left( 2^{\left(\frac{2km}{n}, m\right)-1} + 2^{\lfloor \frac{m}{2} \rfloor} -1 \right).$$
\end{cor}
\begin{proof}
Since $\C{E}(\B{D}_n) = \Psi_{\B{D}_n}(1)$, we have

\begin{align*} \C{E}(\B{D}_n)  =& \frac{1}{2n}
 \sum_{m|n} \left[-2n\frac{n}{m}\,\mu\left(\frac{n}{m}\right)
\left( 2^{\lfloor \frac{m}{2} \rfloor}
 -1\right)+ \mu\left(\frac{n}{m}\right)  \sum_{k=1}^n \sum_{\ell=1}^\frac{n}{m}
\left(F_{\alpha_k}(\ell, m) + F_{\beta_k}(\ell, m)\right)\right] \\
=& \frac{1}{2n}
 \sum_{m|n} \left[-2n\frac{n}{m}\,\mu\left(\frac{n}{m}\right)
\left( 2^{\lfloor \frac{m}{2} \rfloor}
 -1\right)\right.\\
& +  \mu\left(\frac{n}{m}\right)  \sum_{\ell=1}^\frac{n}{m}
\left((2, \frac{n}{m})m \left(2^{\lfloor \frac{m}{2} \rfloor + \left(\frac{2km}{n}, m\right)} + 2^{\lfloor \frac{m}{2} \rfloor + \lfloor \frac{m+2}{2} \rfloor}\right)  \right.\\
 &   +\left.  \left.
\left( n -(2, \frac{n}{m})m \right)2^{\lfloor \frac{m}{2} \rfloor+1} - 2n \right)\right] \\
= &\frac{1}{2n}
 \sum_{m|n} \left[-2n\frac{n}{m}\,\mu\left(\frac{n}{m}\right)
\left( 2^{\lfloor \frac{m}{2} \rfloor}
 -1\right)\right.\\
&     +\left.  \frac{n}{m}\mu\left(\frac{n}{m}\right)
\left((2, \frac{n}{m})m 2^{\lfloor \frac{m}{2} \rfloor} \left(2^{\left(\frac{2km}{n}, m\right)} + 2^{\lfloor \frac{m+2}{2} \rfloor}
-2 \right) +
 2 n 2^{\lfloor \frac{m}{2} \rfloor} - 2n \right)\right] \\
 =&  \sum_{m|n} (2, \frac{n}{m}) \mu\left(\frac{n}{m}\right) 2^{\lfloor \frac{m}{2} \rfloor}
\left( 2^{\left(\frac{2km}{n}, m\right)-1} + 2^{\lfloor \frac{m}{2} \rfloor} -1 \right).
\end{align*}
\end{proof}

\section{Degree distribution polynomials for some circulant graphs}\label{DistC}

In this section, we compute the degree distributions polynomials for the weak equivalence classes
of some circulant graphs of prime power in subsection~\ref{pp}, of order $4p$ in subsection~\ref{4pc}
and of square free order in subsection~\ref{sqfree}.

\subsection{Degree distribution polynomials for circulant graphs of prime power order}\label{pp}

For a prime $p$, let $\B{Z}_{p^m}$ be the cyclic group of order $p^m$.
For each $k=0,1,2, \ldots, m-1$, let $A_k=\{ s \in \B{Z}_{p^m} : (s,p^m)=p^k\}$,
 where $(s,t)$ is the greatest common divisor of the positive integers $s$ and $t$.
Indeed, $A_0$ is $\Aut(\B{Z}_{p^m})$, the set of all automorphisms of
$\B{Z}_{p^m}$. Then $|A_k|= \phi(p^{m-k})=p^{m-k-1}\phi(p)$, where $\phi$ is the Euler pi-function. Note that
 $$G(\C{\B{Z}_{p^m}})=\{ \Omega: -\Omega = \Omega,  \Omega\subset \bigcup_{k=0}^{m-1}A_k \} - \{ \Omega: -\Omega = \Omega,  \Omega\subset \bigcup_{k=1}^{m-1}A_k \}. $$
For each $k=0,1,\ldots, m-1$, let $X_k=\{ \Omega: \Omega^{-1}=-\Omega=\Omega, \Omega\subset A_k\}$.
Now $X_k$ is an invariant subset of $\Aut(\B{Z}_{p^m})$-action.
For  an automorphism  $\alpha$   in $\Aut(\B{Z}_{p^m})$, we observe that
 $$\sum_{\Omega \in \bigcup_{k=0}^{m-1}A_k, -\Omega = \Omega, \alpha(\Omega)=\Omega} x^{|\Omega|}=
\left[\prod_{k=0}^{m-1} \left(1+\sum_{\Omega\in X_k, \alpha(\Omega)=\Omega} x^{|\Omega|}
\right)\right]-1,$$  and
$$\sum_{\Omega \in \bigcup_{k=1}^{m-1}A_k, -\Omega = \Omega, \alpha(\Omega)=\Omega} x^{|\Omega|}=
\left[\prod_{k=1}^{m-1} \left(1+\sum_{\Omega\in X_k, \alpha(\Omega)=\Omega} x^{|\Omega|}
\right)\right]-1.$$
Thus, we have the following lemma.

\begin{lem}\label{fixalpha}
Let  $\alpha$ be an automorphism in $\Aut(\B{Z}_{p^m})$. Then we have
$$\sum_{\Omega\in \Fix_\alpha} x^{|\Omega|}=
\left(\sum_{\Omega\in X_0, \alpha(\Omega)=\Omega} x^{|\Omega|}\right)
\prod_{k=1}^{m-1} \left(1+\sum_{\Omega\in X_k, \alpha(\Omega)=\Omega} x^{|\Omega|}
\right).$$
\end{lem}

Now, we aim to compute $\sum_{\Omega\in X_k, \alpha(\Omega)=\Omega} x^{|\Omega|}$ for each $\alpha\in \Aut(\C{\B{Z}_{p^m}})$
and each $k=0,1, \ldots, m-1$.
First we consider $p=2$.  Note that $\Aut(\B{Z}_{2^m})$ is isomorphic to $\B{Z}_2\times \B{Z}_{2^{m-2}}$ and that
the multiplicative group $A_0$ is equal to the group  $\Aut(\B{Z}_{2^m})$. Let $\eta:A_0 \to \B{Z}_2\times \B{Z}_{2^{(m-2)}}$
be the isomorphism. Then $\eta(-1)=(1,0)$ and hence,
it naturally induces an isomorphism $\bar{\eta}: A_0/\B{Z}_2 \to \B{Z}_{2^{(m-2)}}$.
So any subset $\Omega$ of $A_0$ satisfying $\Omega=-\Omega$ corresponds to a subset of $\B{Z}_{2^{(m-2)}}$ and vice versa via an isomorphism $\bar{\eta}$, namely
there is an one to one correspondence between $X_0$ and $\C{P}(\B{Z}_{2^{(m-2)}})$, where $\C{P}(\B{Z}_{2^{(m-2)}})$ is
the powerset of $\B{Z}_{2^{(m-2)}}$. Furthermore, $\Aut(\B{Z}_{2^m})/\B{Z}_2$-action
on $X_0$ is equivalent to the natural $\B{Z}_{2^{(m-2)}}$-action on $\C{P}(\B{Z}_{2^{(m-2)}})$.
Therefore, for each $\alpha \in A_0=\Aut(\B{Z}_{2^m})$, we can see that
$$\sum_{\Omega\in X_0, \alpha(\Omega)=\Omega} x^{|\Omega|} = \sum_{\bar\Omega\in \bar{\eta}(X_0),\, \bar{\eta}(\alpha)(\bar\Omega)=\bar\Omega}
x^{2|\bar\Omega|}=
\left(1+x^{\frac{2^{m-1}}{\left(\bar{\eta}(\alpha),\, 2^{m-2}\right)}}\right)^{\left(\bar{\eta}(\alpha),\, 2^{m-2}\right)}-1$$
and $\sum_{\Omega\in X_{m-1}, \alpha(\Omega)=\Omega} x^{|\Omega|} = x$.
By a method similar to the case $k=0$, for each $k=1,2, \ldots, m-2$, we can see that
$$\sum_{\Omega\in X_k, \alpha(\Omega)=\Omega} x^{|\Omega|} =
\left(1+x^{\frac{2^{m-k-1}}{\left(\bar{\eta}(\alpha),\, 2^{m-k-2}\right)}}\right)^{\left(\bar{\eta}(\alpha),\, 2^{m-k-2}\right)}-1.$$

Now, by Theorem~\ref{formul1} and Lemma~\ref{fixalpha},

\begin{align*}
&2^{m-1}~ \Psi_{\B{Z}_{2^m}}^w(x) = \displaystyle
 \sum_{\alpha\in {\rm{Aut}}(\B{Z}_{2^m})} \sum_{\Omega\in {\rm{Fix}}_\alpha} x^{|\Omega|} \\
&=  \sum_{\alpha\in {\rm{Aut}}(\B{Z}_{2^m})}
\left[(1+x^{\frac{2^{m-1}}{\left(\bar{\eta}(\alpha),\, 2^{m-2}\right)}})^{\left(\bar{\eta}(\alpha),\, 2^{m-2}\right)}-1\right]
\left(\prod_{k=1}^{m-2}\left(1+x^{\frac{2^{m-k-1}}{\left(\bar{\eta}(\alpha),\, 2^{m-k-2}\right)}}\right)^{\left(\bar{\eta}(\alpha),\, 2^{m-k-2}\right)}\right)(1+x)
 \\
&= 2 \sum_{a\in\B{Z}_{2^{m-2}}}
\left[(1+x^{\frac{2^{m-1}}{\left(a,\, 2^{m-2}\right)}})^{\left(a,\, 2^{m-2}\right)}-1\right](1+x)
\prod_{k=1}^{m-2}\left(1+x^{\frac{2^{m-k-1}}{\left(a,\, 2^{m-k-2}\right)}}\right)^{\left(a,\, 2^{m-k-2}\right)}
 \\
&= 2  \sum_{d|2^{m-2}} \phi(2^{m-2}/d)
\left[\left(1+x^{\frac{2^{m-1}}{d}}\right)^{d}-1\right](1+x)
\prod_{k=1}^{m-2}\left(1+x^{\frac{2^{m-k-1}}{\left(d,\, 2^{m-k-2}\right)}}\right)^{\left(d,\, 2^{m-k-2}\right)} \\
&= 2  \sum_{\ell=0}^{m-2} \phi(2^{m-\ell-2})
\left[\left(1+x^{2^{m-\ell-1}}\right)^{2^\ell}-1\right](1+x)
\prod_{k=1}^{m-2}\left(1+x^{\frac{2^{m-k-1}}{\left(2^\ell,\, 2^{m-k-2}\right)}}\right)^{\left(2^\ell,\, 2^{m-k-2}\right)}.
\end{align*}

We summarize the above discussions as follow.

\begin{thm}\label{2^m}
For each $m\ge 2$, we have
\begin{align*}
& 2^{m-2} \Psi_{\B{Z}_{2^m}}^w(x) \\
& = \sum_{t=0}^{m-2} \phi(2^t)
\left(\left(1+x^{2^{t+1}}\right)^{2^{m-t-2}}-1\right)(1+x)
\prod_{s = t+1}^{m-2} \left(1+x^2\right)^{2^{m-s-2}}
\prod_{s=1}^{t}\left(1+x^{2^{t-s+1}}\right)^{2^{m-t-2}},
\end{align*}
and hence
\begin{align*}
\C{E}^w(\B{Z}_{2^m}) &= \Psi_{\B{Z}_{2^m}}^w(1)=\frac{1}{2^{m-2}}  \sum_{t=0}^{m-2}\phi(2^t)
\left(2^{1+ 2^{m-t-2}}-2\right)
\prod_{s =t+1}^{m-2} 2^{2^{m-s-2}}
\prod_{s=1}^{t}2^{2^{m-t-2}} \\
    &= \frac{1}{2^{m-2}}\left[ 2^{2^{m-1}}(2^{2^{m-2}}-1) +  \sum_{t=1}^{m-2}2^{(t+1)2^{m-t-2}+ t-1}
\left(2^{2^{m-t-2}}-1\right)  \right],
\end{align*}
where the product of the empty index set is defined to be $1$.
\end{thm}

Next we consider the case when $p$ is an odd prime. It is well-known that there is an isomorphism
$\theta :\Aut(\B{Z}_{p^m})\to \B{Z}_{p^{m-1}(p-1)}$. Since $\theta(-1)=\frac{p^{m-1}(p-1)}{2}$, we have an isomorphism
$\bar{\theta}: \Aut(\B{Z}_{p^m})/\B{Z}_2 \to \B{Z}_{p^{m-1}(p-1)}/\B{Z}_2 = \B{Z}_{\frac{p^{m-1}(p-1)}{2}}.$
It is also well-known that the multiplicative group $A_0$ is isomorphic to $\Aut(\B{Z}_{p^m})$.
By a method similar to the case $p=2$, there is an one to one correspondence between $X_0$
and $\C{P}(\B{Z}_{\frac{p^{m-1}(p-1)}{2}})$, where $\C{P}(\B{Z}_{\frac{p^{m-1}(p-1)}{2}})$ is
the powerset of $\B{Z}_{\frac{p^{m-1}(p-1)}{2}}$.
Furthermore, $\Aut(\B{Z}_{p^m})/\B{Z}_2$-action on $X_k$ is equivalent
to the natural $\B{Z}_{\frac{p^{m-1}(p-1)}{2}}$-action on $\C{P}\left(\B{Z}_{\frac{p^{m-k-1}(p-1)}{2}}\right)$
  for each $k=0,1,2,\ldots, m-1$. For each $\alpha\in A_0=\Aut(\B{Z}_{p^m})$, we can see that
$$\sum_{\Omega\in X_k, \alpha(\Omega)=\Omega} x^{|\Omega|} =
\left(1+x^{\frac{p^{m-k-1}(p-1)}{\left(\bar{\theta}(\alpha),\, \frac{p^{m-k-1}(p-1)}{2}\right)}
}\right)^{\left(\bar{\theta}(\alpha),\, \frac{p^{m-k-1}(p-1)}{2}\right)}-1.$$
By combining the above equation, Theorem~\ref{formul1} and Lemma~\ref{fixalpha}, we obtain the following theorem.

\begin{thm}\label{p^m} For each odd prime $p$, we have
\begin{align*}
&\Psi_{\B{Z}_{p^m}}^w(x) \\
&= \frac{2}{ p^{m-1}(p-1) }  \sum_{d|\frac{\phi(p^m)}{2} } \phi\left( \frac{\phi(p^m)}{2d}\right)
\left(\left(1+x^{\frac{\phi(p^m)}{d} }\right)^{d}-1\right)
\prod_{k=1}^{m-1}
\left(1+x^{\frac{\phi\left(p^{m-k}\right)}
{\left(d, \frac{\phi\left(p^{m-k}\right)}{2}\right)}}\right)^{\left(d, \frac{\phi\left(p^{m-k}\right)}{2}\right),}
\end{align*}
and hence
 $$\C{E}^w(\B{Z}_{p^m}) = \Psi_{\B{Z}_{p^m}}^w(1)= \frac{2}{ p^{m-1}(p-1) }
 \displaystyle \sum_{d|\frac{\phi(p^m)}{2} } \phi\left( \frac{\phi(p^m)}{2d}\right)
\left(2^{d}-1\right)
\prod_{k=1}^{m-1}
2^{\left(d, \frac{\phi\left(p^{m-k}\right)}{2}\right)}.
$$
\end{thm}

By a similar method, one can get the following result.

\begin{thm}\label{2p^m} For each odd prime $p$, we have
\begin{align*}
\Psi_{\B{Z}_{2p^m}}^w(x)
&= \frac{2}{ p^{m-1}(p-1) }  \sum_{d|\frac{\phi(p^m)}{2} } \phi\left( \frac{\phi(p^m)}{2d}\right) \left[
\left(\left(1+x^{\frac{\phi(p^m)}{d} }\right)^{d}-1\right)\left(1+x^{\frac{\phi(p^m)}{d} }\right)^{d}  \right.\\
& \times \left(\prod_{k=1}^{m-1}
\left(1+x^{\frac{\phi\left(p^{m-k}\right)}
{\left(d, \frac{\phi\left(p^{m-k}\right)}{2}\right)}}\right)^{2\left(d, \frac{\phi\left(p^{m-k}\right)}{2}\right)} \right)(1+x) \\
&+ \left(\left(1+x^{\frac{\phi(p^m)}{d} }\right)^{d}-1\right)\left\{\left(\prod_{k=1}^{m-1}
\left(1+x^{\frac{\phi\left(p^{m-k}\right)}
{\left(d, \frac{\phi\left(p^{m-k}\right)}{2}\right)}}\right)^{\left(d, \frac{\phi\left(p^{m-k}\right)}{2}\right)} \right)(1+x)-1 \right\} \\
& \left.
\times \prod_{k=1}^{m-1}
\left(1+x^{\frac{\phi\left(p^{m-k}\right)}
{\left(d, \frac{\phi\left(p^{m-k}\right)}{2}\right)}}\right)^{\left(d, \frac{\phi\left(p^{m-k}\right)}{2}\right)}  \right]
\end{align*}
and hence
\begin{align*}
\C{E}^w(\B{Z}_{2p^m})  = \Psi_{\B{Z}_{2p^m}}^w(1) & = \frac{2}{ p^{m-1}(p-1) }
 \displaystyle \sum_{d|\frac{\phi(p^m)}{2} } \phi\left( \frac{\phi(p^m)}{2d}\right) \left[
2^{d+1}\left(2^{d}-1\right)
\prod_{k=1}^{m-1}
2^{2\left(d, \frac{\phi\left(p^{m-k}\right)}{2}\right)} \right.\\
& \left. + \left(2^d-1 \right)\left( 2\prod_{k=1}^{m-1}
2^{\left(d, \frac{\phi\left(p^{m-k}\right)}{2}\right)} -1\right)\prod_{k=1}^{m-1}
2^{\left(d, \frac{\phi\left(p^{m-k}\right)}{2}\right)} \right].
\end{align*}
\end{thm}

\subsection{Degree distribution polynomials for circulant graphs of order $4p$}\label{4pc}

In this subsection, we find the degree distribution polynomials for circulant graphs of  order $4p$, where $p$ is an odd prime number.

Let $n=4p$ for some prime number $p$. For our convenience, we consider  $\B{Z}_n$ as $\B{Z}_{4} \times \B{Z}_{p}$.
Now $\Aut({\B{Z}_n})$ is isomorphic to $\B{Z}^{*}_{4p} \simeq (\B{Z}^{*}_{4} \times \B{Z}^{*}_{p}) \simeq (\B{Z}_{2} \times \B{Z}_{p-1})$.
So any automorphism $ \sigma\in \Aut({\B{Z}_n})$ can be identified with an
element $(j, k)$ in $\B{Z}_{2} \times \B{Z}_{p-1}$.
Let us consider $\B{Z}_{4} \times \B{Z}_{p}$ as a disjoint union of the following subsets;
$$\begin{array}{rl}
A_i &= \{(n_1,n_2 ) \ | \ n_1 =i \ \ \mbox{and} \ \ n_2 \neq 0 \ \}, \\
B_i &= \{(n_1,n_2 ) \ | \  n_1 =i \ \ \mbox{and} \ \ n_2 = 0 \ \},
\end{array}$$
where $i=0,1,2,3$. Note that $\Aut(\B{Z}_{4} \times \B{Z}_{p})$ can also be identified by $A_1 \cup A_3$ and all of the subsets $A_1 \cup A_3$, $A_0$, $A_2$, $B_1 \cup B_3$, $B_0$, $B_2$ are closed under $\Aut(\B{Z}_{4} \times \B{Z}_{p})$-action.

Let $\sigma$ be an element in $\Aut({\B{Z}_n})$ whose corresponding element $\B{Z}_{2} \times \B{Z}_{p-1}$ is $(j, k)$ with $(k,p-1)=d$.
Now an orbit $O_1$ of $\sigma$ in $A_1 \cup A_3$ satisfies $-O_1 = O_1$ if and only if $j=1$, $d |  \frac{p-1}{2}$ and $\frac{p-1}{2d}$ is odd.
Note that if $j=1$ and $\frac{p-1}{d}$ is odd then $|O_2|=\frac{2p-2}{d}$, and if $j=1$, $d |  \frac{p-1}{2}$ and $\frac{p-1}{2d}$ is odd then $|O_2|=\frac{p-1}{d}$.
For any $i=0,2$ and for any orbit $O_2$ of $\sigma$ in $A_i$, we have $-O_2 = O_2$ if and only if $d |  \frac{p-1}{2}$.

For any subset $\Omega \subseteq \B{Z}_{4} \times \B{Z}_{p}$, $\Omega$ is a generating set of $\B{Z}_{4} \times \B{Z}_{p}$ if and only if $\Omega \cap (A_1 \cup A_3) \neq \emptyset$ or $(\Omega \cap (A_0 \cup A_2) \neq \emptyset$ and $\Omega \cap (B_1 \cup B_3) \neq \emptyset)$. So we have the following theorem.

\begin{thm}\label{4p} For each odd prime $p$, we have
\begin{align*}
\Psi_{\B{Z}_{4p}}^w(x) &=  \frac{1}{2p-2} \sum_{d|p-1}\phi\left(\frac{p-1}{d}\right) \left[
\left\{\left(1+x^{\frac{2p-2}{d}}\right)^{d}
 +\left(1+x^{\frac{2p-2}{\left(\frac{p-1}{2}+d, p-1\right)}}\right)^{\left(\frac{p-1}{2}+d, p-1\right)} \right\}\right. \\
& \times \left.  \left(1+x^{\frac{p-1}{\left(d, \frac{p-1}{2}\right)}}\right)^{2\left(d, \frac{p-1}{2}\right)} (1+x^2)(1+x)\right.\\
& \left. -2\left(1+x^{\frac{p-1}{\left(d, \frac{p-1}{2}\right)}}\right)^{2\left(d, \frac{p-1}{2}\right)}(1+x)
-2(1+x^2)(1+x) +2(1+x)  \right],
\end{align*}
and hence
 $$ \C{E}^w(\B{Z}_{4p}) = \Psi_{\B{Z}_{4p}}^w(1)= \frac{1}{2p-2}
 \displaystyle \sum_{d|p-1}\phi\left(\frac{p-1}{d}\right)
\left[ 4^{\left(d, \frac{p-1}{2}\right)+1} \left( 2^d +2^{\left(\frac{p-1}{2}+d, p-1\right)} \right)
- 2^{2\left(d, \frac{p-1}{2}\right)+2} -4 \right].
$$
\end{thm}

Note that $\left(\frac{p-1}{2}+d, p-1\right)$ is $\frac{d}{2}$ if $\frac{p-1}{d}$ is odd; $d$ if both $\frac{p-1}{d}$ and $\frac{p-1}{2d}$ are even; $2d$ if  $\frac{p-1}{d}$ is even but $\frac{p-1}{2d}$ is odd.

\subsection{Degree distribution polynomials for circulant graphs of square free order}\label{sqfree}

In this subsection, we find the degree distribution polynomials for circulant graphs of square free order.
First we consider the case that $n$ is a product of three distinct primes; (1) $n=2p_1p_2 \ (2 < p_1 < p_2 )$
and (2) $n=p_1p_2p_3  \ (2<p_1 < p_2 < p_3)$.

First, let $n=2p_1p_2 \ (2 < p_1 < p_2 )$ be a square free natural number.
Since $\Aut({\B{Z}_n})$ is isomorphic to $\B{Z}_{p_1-1} \times \B{Z}_{p_2-1}$,
any automorphism $ \sigma\in \Aut({\B{Z}_n})$ can be identified with an
element $(n_1, n_2)$ in $\B{Z}_{p_1-1} \times \B{Z}_{p_2-1}$.
For $\sigma \in \Aut({\B{Z}_n})$ whose corresponding element in
$\B{Z}_{p_1-1} \times \B{Z}_{p_2-1}$ is $(n_1,n_2)$.
Let $d_1 = (n_1, p_1-1)$, $d_2 = (n_2,p_2-1)$, $\beta_1 =(d_1, \frac{p_1-1}{2})$ and
$\beta_2 =(d_2, \frac{p_2-1}{2})$.
Now we have

$$\sum_{\Omega\in S(\B{Z}_n),\,\sigma(\Omega)= \Omega} =
(1+x)\left(1+x^{\frac{p_1-1}{(d_1, \frac{p_1-1}{2})}}\right)^{2 \beta_1}
\left(1+x^{\frac{p_2-1}{(d_2, \frac{p_2-1}{2})}}\right)^{2 \beta_2}
\left(1+x^{\alpha(d_1,d_2)}\right)^{\frac{2(p_1-1)(p_2-1)}{\alpha(d_1,d_2)}},$$
where $\alpha(d_1 ,d_2 )$ is defined as 
$$\alpha(d_1, d_2) = \begin{cases}
 {\rm{lcm}}(\frac{p_1-1}{d_1},\frac{p_2-1}{d_2}) & {\rm{if}} ~d_i|\frac{p_i-1}{2}~ {\rm{and}~\rm{there}~\rm{exists}~\rm{a}~\rm{constant}}~c~{\rm{such}}\\
 & \mbox{ that  $cd_i \equiv \frac{p_i-1}{2} \  \left({\rm mod} ~p_i-1\right)$ for any $i=1, 2$,}\\[0.5ex]
2\,\mbox{lcm }(\frac{p_1-1}{d_1},\frac{p_2-1}{d_2})
& \mbox{otherwise.}\end{cases}$$ 
Let $\gamma = \frac{(p_1-1)(p_2-1)}{\alpha(d_1,d_2)}$ then, we have
\begin{align*}
\Psi_{\B{Z}_n}^w(x) &=\frac{1}{|\Aut(\B{Z}_n)|} \sum_{\sigma\in {\rm{Aut}}(\B{Z}_n)}
\left(\sum_{\B{Z}_m\le \B{Z}_n} \mu(\B{Z}_m) \left(\sum_{\Omega\in
S(\B{Z}_m),\,\sigma(\Omega)=\Omega}x^{|\Omega|}\right)\right) \\
&= \frac{1}{(p_1-1)(p_2-1)} \sum_{d_1|p_1-1,~d_2|p_2-1} \phi
\left(\frac{p_1-1}{d_1}\right)\phi\left(\frac{p_2-1}{d_2}\right)  \\
& \times \left[ (1+x)\left(1+x^{\frac{p_1-1}{(d_1, \frac{p_1-1}{2})}}
\right)^{2\beta_1}\left(1+x^{\frac{p_2-1}{(d_2, \frac{p_2-1}{2})}}\right)^{2\beta_2}
\left(1+x^{\alpha(d_1,d_2)}\right)^{ \gamma} \right. \\
&- (1+x)\left(1+x^{\frac{p_1-1}{(d_1, \frac{p_1-1}{2})}}\right)^{2\beta_1} -
(1+x)\left(1+x^{\frac{p_2-1}{(d_2, \frac{p_2-1}{2})}}\right)^{2\beta_2}\\
&- \left(1+x^{\frac{p_1-1}{(d_1, \frac{p_1-1}{2})}}\right)^{\beta_1}
\left(1+x^{\frac{p_2-1}{(d_2, \frac{p_2-1}{2})}}\right)^{\beta_2}
\left(1+x^{\alpha(d_1,d_2)}\right)^{ \gamma} \\
& \left. + (1+x) + \left(1+x^{\frac{p_1-1}{(d_1, \frac{p_1-1}{2})}}\right)^{\beta_1}
+\left(1+x^{\frac{p_2-1}{(d_2, \frac{p_2-1}{2})}}\right)^{\beta_2} -1 \right].
\end{align*}

For the next case, let $n=p_1p_2p_3  \ (2<p_1 < p_2 < p_3)$ be a square free natural number.
Since $\Aut({\B{Z}_n})$ is isomorphic to $\B{Z}_{p_1-1} \times \B{Z}_{p_2-1}\times \B{Z}_{p_3-1}$,
any automorphism $\sigma\in \Aut({\B{Z}_n})$ can be identified with an
element $(n_1, n_2, n_3)$ in $\B{Z}_{p_1-1} \times \B{Z}_{p_2-1}\times \B{Z}_{p_3-1}$.
For $\sigma \in \Aut({\B{Z}_n})$ whose corresponding element in
$\B{Z}_{p_1-1} \times \B{Z}_{p_2-1}\times \B{Z}_{p_3-1}$ is $(n_1,n_2,n_3)$,
let $d_1 = (n_1, p_1-1)$, $d_2 = (n_2,p_2-1)$, $d_3 = (n_3, p_3-1)$, $\beta_1 =(d_1, \frac{p_1-1}{2})$,
$\beta_2 =(d_2, \frac{p_2-1}{2})$, $\beta_3 =(d_3, \frac{p_3-1}{2})$,
$\gamma_{12} = \frac{(p_1-1)(p_2-1)}{\alpha(d_1,d_2)}$,
$\gamma_{13} = \frac{(p_1-1)(p_3-1)}{\alpha(d_1,d_3)}$ and
$\gamma_{23} = \frac{(p_2-1)(p_3-1)}{\alpha(d_2,d_3)}$.  Let $\alpha(d_1 ,d_2, d_3)$ be defined as
$$\alpha(d_1, d_2, d_3) = \begin{cases}
 \mbox{lcm }(\frac{p_1-1}{d_1},\frac{p_2-1}{d_2},\frac{p_3-1}{d_\ell}) & \mbox{if $d_i|\frac{p_i-1}{2}$  and
 }~\mbox{there exists a constant $c$ such}\\
 & \mbox{ that  $cd_i \equiv \frac{p_i-1}{2} \  \left({\rm mod} ~p_i-1\right)$ for any $i=1, 2, 3$,}\\[0.5ex]
2\,\mbox{lcm }(\frac{p_1-1}{d_1},\frac{p_2-1}{d_2},\frac{p_3-1}{d_3})
& \mbox{otherwise.}\end{cases}$$  
Let $\delta = \frac{(p_1-1)(p_2-1)(p_3-1)}{\alpha(d_1, d_2, d_3)}$.
Now we have
\begin{align*}
\sum_{\Omega\in S(\B{Z}_n),\,\sigma(\Omega)= \Omega} &=
\left(1+x^{\frac{p_1-1}{(d_1, \frac{p_1-1}{2})}}\right)^{\beta_1}
\left(1+x^{\frac{p_2-1}{(d_2, \frac{p_2-1}{2})}}\right)^{\beta_2}
\left(1+x^{\frac{p_3-1}{(d_3, \frac{p_3-1}{2})}}\right)^{\beta_3} \\
 &\times \left(1+x^{\alpha(d_1,d_2)}\right)^{ \gamma_{12}}
 \left(1+x^{\alpha(d_1,d_3)}\right)^{\gamma_{13}}
 \left(1+x^{\alpha(d_2,d_3)}\right)^{\gamma_{23}}
 \left(1+x^{\alpha(d_1,d_2,d_3)}\right)^{\delta}.
\end{align*}

Therefore,

\begin{align*}
\Psi_{\B{Z}_n}^w(x) &=\frac{1}{|\Aut(\B{Z}_n)|} \sum_{\sigma\in {\rm{Aut}}(\B{Z}_n)}
\left(\sum_{\B{Z}_m\le \B{Z}_n} \mu(\B{Z}_m) \left(\sum_{\Omega\in S(\B{Z}_m),\,\sigma(\Omega)=\Omega}x^{|\Omega|}\right)\right) \\
&= \frac{1}{(p_1-1)(p_2-1)(p_3-1)} \sum_{d_i|p_i-1,~i=1,2,3} \phi\left(\frac{p_1-1}{d_1}\right)\phi\left(\frac{p_2-1}{d_2}\right)\phi\left(\frac{p_3-1}{d_3}\right)  \\
& \times \left[  \left(1+x^{\frac{p_1-1}{(d_1, \frac{p_1-1}{2})}}\right)^{\beta_1}
\left(1+x^{\frac{p_2-1}{(d_2, \frac{p_2-1}{2})}}\right)^{\beta_2}
\left(1+x^{\frac{p_3-1}{(d_3, \frac{p_3-1}{2})}}\right)^{\beta_3} \right. \\
 &\times \left(1+x^{\alpha(d_1,d_2)}\right)^{ \gamma_{12}}
 \left(1+x^{\alpha(d_1,d_3)}\right)^{ \gamma_{13}}
 \left(1+x^{\alpha(d_2,d_3)}\right)^{ \gamma_{23}}
 \left(1+x^{\alpha(d_1,d_2,d_3)}\right)^{ \delta} \\
&\left. - \sum_{1 \le i < j \le 3} \left(1+x^{\frac{p_i-1}{\beta_i}}\right)^{\beta_i}
\left(1+x^{\frac{p_j-1}{\beta_j}}\right)^{\beta_j}
\left(1+x^{\alpha(d_i,d_j)}\right)^{ \gamma_{ij}}
 + \sum_{i=1}^{3}\left(1+x^{\frac{p_i-1}{\beta_i}}\right)^{\beta_i} -1 \right].
\end{align*}

From now on, we consider general case. 
Let $n$ be an odd square free number, namely $n=p_1 p_2 \cdots p_\ell$ 
for some distinct prime numbers $p_1, p_2, \ldots, p_\ell$.
For each divisor $d_i$ of $p_{i}-1$ for $i=1, 2, \ldots, \ell$, we define
$f(d_1, d_2, \ldots, d_\ell)$ be the polynomial
$$\left(1+x^{\alpha(d_1, d_2, \ldots, d_\ell)}\right)^{
{\frac{(p_1-1)(p_2-1)\cdots(p_\ell-1)}{\alpha(d_1, d_2, \ldots, d_\ell)}}},$$
 where $\alpha(d_1, d_2, \ldots, d_\ell)$ is the number defined as follows:
 $$\begin{cases}
  \frac{p_1-1}{ (d_1, \frac{p_1-1}{2})} &\mbox{if $\ell=1$,}\\[1.5ex]
\mbox{lcm }(\frac{p_1-1}{d_1},\frac{p_2-1}{d_2},\ldots,\frac{p_\ell-1}{d_\ell}) & \mbox{if $\ell \ge 2$, $d_i|\frac{p_i-1}{2}$,
 }~\mbox{there exists a constant $c$ such that}\\
 & \mbox{  $cd_i \equiv \frac{p_i-1}{2} \ \  \left({\rm mod} ~p_i-1\right)$ for any $i=1, \dots, \ell$,}\\[0.5ex]
2\,\mbox{lcm }(\frac{p_1-1}{d_1},\frac{p_2-1}{d_2},\ldots,\frac{p_\ell-1}{d_\ell})
& \mbox{otherwise.}\end{cases}
$$
For convenience, for each $d_i$ such that $d_i|p_{i}-1$ for $i=1,2,\ldots,\ell$, we define
$$F_n(d_1, d_2, \ldots, d_\ell) =\prod_{k=0}^{l-1} \prod_{1\le i_i < i_2 < \ldots < i_k \le \ell}
f(d_1, d_2, \ldots, \widehat{d_{i_1}}, \ldots, \widehat{d_{i_2}},\ldots, \widehat{d_{i_k}}\ldots, d_\ell),
$$
where $$f(d_1, \ldots, \widehat{d_{i_1}}, \ldots,  \widehat{d_{i_k}}\ldots, d_\ell)=
f(d_1,\ldots, d_{i_1-1}, d_{i_1+1} \ldots, d_{i_k-1},d_{i_k+1}\ldots, d_\ell).$$

Since $\Aut({\B{Z}_n})$ is isomorphic to $\B{Z}_{p_1-1} \times \B{Z}_{p_2-1} \times \cdots\times\B{Z}_{p_\ell-1}$,
any automorphism $ \sigma\in \Aut({\B{Z}_n})$ can be identified with an
element $(n_1, n_2, \ldots, n_\ell)$ in $\B{Z}_{p_1-1} \times \B{Z}_{p_2-1} \times \cdots\times\B{Z}_{p_\ell-1}$.
Thus, we can see that
$$\sum_{\sigma\in {\rm{Aut}}(\B{Z}_n)} \sum_{\Omega\in S(\B{Z}_n),\,\sigma(\Omega)= \Omega} x^{|\Omega|}=
\sum_{(d_1, d_2 , \ldots, d_l) \in \Phi} \left(\prod_{i=1}^{\ell}\phi\left(\frac{p_i-1}{d_i}\right)\right)
F_n(d_1, \ldots, d_\ell),$$
where  $\Phi = \{ (d_1, d_2 , \ldots, d_l) :~ d_i|(p_i-1), 1\le i\le \ell \}$.
Similarly one can see that for any
$m=p_1 \cdots p_{i_1-1} p_{i_1+1} \cdots$ $ p_{i_2-1}$ $p_{i_2+1}$ $\cdots$ $ p_{i_k-1}$ $ p_{i_k+1}$ $\cdots$ $p_\ell$,
we have
\begin{align*}
& \sum_{\sigma\in \Aut(\B{Z}_n)} \hspace{.3cm}\sum_{\Omega\in S(\B{Z}_m),\,\sigma(\Omega)= \Omega} x^{|\Omega|} \\
& \hspace{1cm} =
\sum_{(d_1, d_2 , \ldots, d_l) \in \Phi} \left( \prod_{i=1}^{\ell}\phi\left(\frac{p_i-1}{d_i}\right) \right)
F_m(d_1, \ldots,  d_{i_1-1} d_{i_1+1}, \ldots ,d_{i_k-1},d_{i_k+1}, \ldots, d_\ell).
\end{align*}

For our convenience, we denote  $F_m(d_1, \ldots,  d_{i_1-1} d_{i_1+1}, \ldots ,d_{i_k-1},d_{i_k+1}, \ldots, d_\ell)$ simply by
$F_{i_1, i_2, \ldots, i_k}(d_1,\ldots, d_\ell)$. Now the following theorem comes form the above discussions, Burnside Lemma and the principal of inclusion and exclusion.

\begin{thm}\label{pqr}
Let $n=p_1p_2 \cdots p_\ell$ be the product of any given $\ell$ distinct odd prime numbers $p_1 < p_2 < \ldots < p_\ell$.
Then we have
$$ \Psi_{\B{Z}_{n}}^w(x)= \frac{1}{\phi(n)} \sum_{(d_1, d_2, \ldots, d_\ell)}
\left(\prod_{i=1}^{\ell} \phi\left(\frac{p_i-1}{d_i}\right) \right)
 \left(\sum_{k=0}^{\ell} (-1)^k \sum_{1\le i_1< i_2 < \ldots < i_k \le \ell} F_{i_1, i_2, \ldots, i_k}(d_1,\ldots, d_\ell) \right),
$$
where $d_i$ runs over all divisors of $p_i-1$ for each $i=1,2, \ldots,\ell$ and  $F_{1, 2, \ldots, \ell}(d_1,\ldots, d_\ell)$ is considered to be $1$.
\end{thm}

\begin{table}
\begin{tabular}{|c|l|c|}\hline
$n$ & $\Psi_{\B{Z}_{n}}(x)$ & $\C{E}(\B{Z}_{n})$ \\
\hline $2$ & $x$ & $1$  \\
\hline $3$ & $x^2$ & $1$  \\
\hline $4$ & $x^2 + x^3$ & $2$  \\
\hline $5$ & $2x^2 + x^4$ & $3$  \\
\hline $6$ & $x^2 + 2x^3 +x^4 + x^5$ & $5$  \\
\hline $7$ & $3x^2+3x^4 +x^6$ & $7$  \\
\hline $8$ & $2x^2+2x^3+3x^4+3x^5+x^6+x^7$ & $12$  \\
\hline $9$ & $3x^2+6x^4+4x^6+x^8$ & $14$  \\
\hline $10$ & $2x^2+4x^3+5x^4+6x^5+4x^6+4x^7+x^8+x^9$ & $27$  \\
\hline $11$ & $5x^2+10x^4+10x^6+5x^8+x^{10}$ & $31$  \\
\hline $12$ &  $2x^2+2x^3+9x^4+9x^5+10x^6+10x^7+5x^8+5x^9 +x^{10}+x^{11}$ & $54$  \\
\hline $13$ & $6x^2+15x^4+20x^6+15x^8+6x^{10}+x^{12}$ & $63$  \\
\hline $14$ & $x^2+2x^3+4x^4+5x^5+7x^6+8x^7+5x^8+5x^9+2x^{10}+2x^{11}+x^{12}+x^{13}$ & $119$  \\
\hline $15$ &  $6x^2+20x^4+35x^6+35x^8+21x^{10}+7x^{12}+x^{14}$ & $125$  \\
\hline $16$ & $\begin{matrix}4x^2+4x^3+18x^4+18x^5+34x^6+34x^7+ 35x^8 \\+35x^9+ 21x^{10} +21x^{11} + 7x^{12}+7x^{13}+ x^{14}+x^{15}\end{matrix}$ & $240$  \\
\hline $17$ & $ 8x^2+28x^4+56x^6+70x^8+56x^{10}+28x^{12}+8x^{14}+x^{16}$ & $255$  \\
\hline $18$ & $\begin{matrix}4x^2+5x^3+22x^4+28x^5 + 52x^6+56x^6+  56x^7+69x^8+70x^9\\
+ 56x^{10} +56x^{11}+ 28x^{12} +28x^{13}+8x^{14} +8x^{15}+x^{16} +x^{17}\end{matrix}$ & $548$ \\
\hline $19$ & $9x^2+36x^4+84x^6+126x^8 + 126x^{10}+84x^{12}+  36x^{14}+9x^{16}+x^{18}$ & $511$  \\
\hline  $20$ & $\begin{matrix}4x^2+4x^3+30x^4+30x^5+78x^6+78x^7+125x^8+125x^9+126x^{10}\\
+126x^{11}+84x^{12}+84x^{13}+36x^{14}+36x^{15}+9x^{16}+9x^{17}+x^{18}+x^{19}\end{matrix}$  & $986$  \\
\hline
\end{tabular}
\vskip .2cm
\caption{The degree distribution polynomials $\Psi_{\B{Z}_{n}}(x)$ for circulant graphs up to equivalence for $n \le 20$.} \label{equitable}
\end{table}

\begin{table}
\begin{tabular}{|c|l|c|}\hline
$n$ & $\Psi_{\B{Z}_{n}}^w(x)$ & $\C{E}^w(\B{Z}_{n})$ \\
\hline $2$ & $x$ & $1$  \\
\hline $3$ & $x^2$ & $1$  \\
\hline $4$ & $x^2+x^3$ & $2$  \\
\hline $5$ & $x^2+x^4$ & $2$  \\
\hline $6$ & $x^2 + 2x^3 +x^4 + x^5$ & $5$  \\
\hline $7$ & $x^2 + x^4 + x^6$ & $3$  \\
\hline $8$ & $x^2+x^3+2x^4+2x^5+x^6+x^7$ & $8$  \\
\hline $9$ & $x^2+2x^4+2x^6+x^8$ & $6$  \\
\hline $10$ & $x^2+2x^3+3x^4+4x^5+2x^6+2x^7+x^8+x^9$ & $16$  \\
\hline $11$ & $x^2+2x^4+2x^6+x^8+x^{10}$ & $7$  \\
\hline $12$ &  $x^2+x^3 +6x^4+6x^5+7x^6+7x^7+4x^8+4x^9+x^{10}+x^{11}$ & $38$  \\
\hline $13$ & $x^2+3x^4+4x^6+3x^8+x^{10}+x^{12}$ & $13$  \\
\hline $14$ &  $x^2+2x^3+4x^4+5x^5+7x^6+8x^7+5x^8+5x^9+2x^{10}+2x^{11}+x^{12}+x^{13}$ & $43$  \\
\hline $15$ &  $x^2+6x^4+11x^6+11x^8+7x^{10}+3x^{12}+x^{14}$ & $40$  \\
\hline $16$ & $\begin{matrix}x^2+x^3+5x^4+5x^5+10x^6+10x^7+11x^8\\+11x^9+7x^{10}+7x^{11}+3x^{12}+3x^{13}+x^{14}+x^{15}\end{matrix}$ & $76$  \\
\hline $17$ & $x^2+4x^4+7x^6+10x^8+7x^{10}+4x^{12}+x^{14}+x^{16}$ & $35$  \\
\hline $18$ & $\begin{matrix}x^2+2x^3+7x^4+9x^5+18x^6+20x^7+25x^8+26x^9+20x^{10}\\+20x^{11}+10x^{12}+10x^{13}+4x^{14}+4x^{15} +x^{16}+x^{17}\end{matrix}$  & $178$  \\
\hline $19$ & $x^2+4x^4+10x^6+14x^8+14x^{10}+10x^{12}+4x^{14}+x^{16}+x^{18}$ & $59$  \\
\hline  $20$ & $\begin{matrix} x^2+x^3+9x^4+9x^5+25x^6+25x^7+38x^8+38x^9+39x^{10}+39x^{11}\\
+27x^{12}+27x^{13}+13x^{14}+13x^{15} +4x^{16}+4x^{17}+x^{18}+x^{19}\end{matrix}$  & $314$  \\
\hline
\end{tabular}
\vskip .2cm
\caption{The degree distribution polynomials $\Psi_{\B{Z}_{n}}(x)$ for circulant graphs up to weak equivalence for $n \le 20$.} \label{wequitable}
\end{table}

For even square free number $n$, we have the following theorem by a similar way.

\begin{thm}\label{2pqr}
Let $n=2p_1p_2 \cdots p_\ell$ be a square free number, where $p_1, p_2,\ldots, p_\ell$ are $\ell$ distinct prime numbers.
Then we have
\begin{align*} \Psi_{\B{Z}_{n}}^w(x) &= \frac{1}{\phi(n)} \sum_{(d_1, d_2, \ldots, d_\ell)}
\left(\prod_{i=1}^{\ell} \phi\left(\frac{p_i-1}{d_i}\right) \right)
 \left[(-1)^{\ell}(1+x) \right. \\
 & \left.+ \sum_{k=0}^{\ell-1} (-1)^k \left( \sum_{1\le i_1< i_2 < \ldots < i_k \le \ell} \left( (1+x)F_{i_1, i_2, \ldots, i_k}(d_1,\ldots, d_\ell)^2 - F_{i_1, i_2, \ldots, i_k}(d_1,\ldots, d_\ell) \right) \right) \right],
\end{align*}
where $d_i$ runs over all divisors of $p_i-1$ for each $i=1,2, \ldots,\ell$.
\end{thm}

For $n \le 20$, the degree distribution polynomials for circulant graphs up to equivalence are 
given in Table~\ref{equitable} and the degree distribution polynomials for circulant graphs up to weak equivalence are
given in Table~\ref{wequitable}.

\section{Degree distribution polynomials for Cayley graphs of order $p$ or $2p$} \label{cayleydis}

Throughout this section, we assume that $p$ is an odd prime. It is well-known that  any group of order $p$ is isomorphic to
the cyclic group $\B{Z}_p$ and that  any group of order $2p$ is
 isomorphic to the cyclic group $\B{Z}_{2p}$ or the dihedral group $\B{D}_p$. By Theorem~\ref{p^m}, we have the following corollary.

\begin{cor}\label{zp} For each odd prime $p$, the  distribution polynomial for the weak equivalence classes of Cayley graphs of order $p$ is
$$\Psi_{\B{Z}_{p}}^w(x)
= \frac{2}{p-1} \displaystyle \sum_{d|\frac{p-1}{2} } \phi\left( \frac{p-1}{2d}\right)
\left(\left(1+x^{\frac{p-1}{d} }\right)^{d}-1\right),
$$ and hence
$$\C{E}^w(\B{Z}_{p}) = \Psi_{\B{Z}_{p}}^w(1)=\frac{2}{p-1} \displaystyle \sum_{d|\frac{p-1}{2} } \phi\left( \frac{p-1}{2d}\right)
\left(2^{d}-1\right).$$
\end{cor}

 By Theorem~\ref{pqr}, we have the following.

\begin{cor}\label{z2p} For each odd prime $p$, we have
$$\Psi_{\B{Z}_{2p}}^w(x)
= \left[\frac{2}{p-1} \sum_{d|\frac{p-1}{2} } \phi\left( \frac{p-1}{2d}\right)
\left(1+x^{\frac{p-1}{d} }\right)^{d} \left(\left(1+x^{\frac{p-1}{d} }\right)^{d}(1+x)-1 \right) \right] -x,
$$ and hence
$$\C{E}^w(\B{Z}_{2p}) = \Psi_{\B{Z}_{2p}}^w(1)=-1 + \frac{2}{p-1} \sum_{d|\frac{p-1}{2} } \phi\left( \frac{p-1}{2d}\right)
2^{d} \left(2^{d+1}
-1\right).$$
\end{cor}

As the final part of this section, we will find the  distribution polynomials for the weak equivalence classes of  Cayley graphs
whose underlying group is the dihedral group $\B{D}_p$. It is well-known that $\Aut(\B{D}_p)$
is the set $\{\alpha_{ij}: i=1,2,\ldots, p-1, j=1,2, \ldots, p\}$, where
$\alpha_{ij}(a)=a^i$ and $\alpha_{ij}(b)=ba^j$. For $i=1$, it is not hard to show that
$$\sum_{\Omega\in S(\B{D}_p),\,\alpha_{1j}(\Omega)=\Omega}x^{|\Omega|}=
\begin{cases}
(1+x^2)^{\frac{p-1}{2}}(1+x)^p-1 & \mbox{ if $j=p$,}\\
(1+x^2)^{\frac{p-1}{2}}(1+x^p)-1 & \mbox{ if $j\not=p$,}
\end{cases}
$$
$$\sum_{\Omega\in S(\B{Z}_p),\,\alpha_{1j}(\Omega)=\Omega}x^{|\Omega|}=
(1+x^2)^{\frac{p-1}{2}}-1,$$
and
$$\sum_{\Omega\in S(\B{D}_1),\,\alpha_{1j}(\Omega)=\Omega}x^{|\Omega|}=
\begin{cases}
x & \mbox{ if $j=p$,}\\
0 & \mbox{ if $j\not=p$.}
\end{cases}
$$
For convenience, for each $i\in \Aut(\B{Z}_p)$, let $i^*$ be the element in $\B{Z}_{p-1}$  corresponding to $i$ under the
isomorphism between $\Aut(\B{Z}_p)$ and $\B{Z}_{p-1}$. Then $\frac{p-1}{(i^*, p-1)}$ and  $\frac{p-1}{(i^*, \frac{p-1}{2})}$ have
 the order of $i^*$ in $\B{Z}_{p-1}$ and $\B{Z}_{p-1}/\B{Z}_2$, respectively,
 where $\B{Z}_2$ is the subgroup $\{0,\frac{p-1}{2}\}$ of $\B{Z}_{p-1}$. For $i\not=1$, it is also not hard to show that
$$\sum_{\Omega\in S(\B{D}_p),\,\alpha_{ij}(\Omega)=\Omega}x^{|\Omega|}=
\left(1+x^\frac{p-1}{(i^*, \frac{p-1}{2})}\right)^{(i^*, \frac{p-1}{2})}
\left(1+x^\frac{p-1}{(i^*, p-1)}\right)^{(i^*, p-1)}(1+x)-1,$$
$$\sum_{\Omega\in S(\B{Z}_p),\,\alpha_{ij}(\Omega)=\Omega}x^{|\Omega|}=
\left(1+x^\frac{p-1}{(i^*, \frac{p-1}{2})}\right)^{(i^*, \frac{p-1}{2})}-1, $$
and
$$\sum_{\Omega\in S(\B{D}_1),\,\alpha_{ij}(\Omega)=\Omega}x^{|\Omega|}=
\begin{cases}
x & \mbox{ if $\B{D}_1=\{0\} \cup \{ba^{(i-1)^{-1}(-j)}\}$,}\\
0 & \mbox{ if otherwise,}
\end{cases}
$$
where $(i-1)^{-1}$ is the inverse in $\Aut(\B{Z}_p)$. Now, Theorem~\ref{dp} follows from the above discussion and
 Theorem~\ref{formul2}.

\begin{thm}\label{dp} For each odd prime $p$, we have
\begin{align*}
p(p-1)&\Psi_{\B{D}_{p}}^w(x) = \displaystyle (1+x^2)^{\frac{p-1}{2}}\left[(1+x)^p+(p-1)x^p-1\right]-px  \\
&+ \sum_{d|\frac{p-1}{2}} \phi\left(\frac{p-1}{2}\right)
\left[\left(1+x^\frac{p-1}{(d, \frac{p-1}{2})}\right)^{(d, \frac{p-1}{2})}
\left(\left(1+x^\frac{p-1}{d}\right)^{d}(1+x)-1\right)-x\right],
\end{align*}
and hence
\begin{align*}
&\C{E}^w(\B{D}_{p})\\
&=\frac{1}{p(p-1)}\left[ 2^{\frac{p-1}{2}}\left(2^p+p-2\right)-p
 + \displaystyle  p\,\sum_{d|\frac{p-1}{2}} \phi\left(\frac{p-1}{2}\right)
\left( 2^{(d, \frac{p-1}{2})+d+1}
- 2^{(d, \frac{p-1}{2})} -1\right) \right].
\end{align*}
\end{thm}

\end{document}